\documentclass[10pt]{article}
\usepackage[a4paper, footskip=15mm,footnotesep=9mm,nohead,dvips]{geometry}
\usepackage{amsmath}
\usepackage{amsfonts}
\usepackage{amssymb}
\usepackage{amsthm}
\usepackage{enumerate}
\usepackage{setspace}
\usepackage[all,cmtip]{xy}
\usepackage{epsfig, lscape}

\numberwithin{equation}{section}
\makeatletter

\theoremstyle{plain}
\newtheorem{theorem}{Theorem}[section]
\newtheorem{lemma}[theorem]{Lemma}
\newtheorem{definition}[theorem]{Definition}
\newtheorem{rem}[theorem]{Remark}
\newtheorem{exa}[theorem]{Example}  
\newtheorem{notation}[theorem]{Notation}

\begin{document}

\title{On The Structure and Automorphism Group  \\ of Finite Alexander Quandles}
\date{July 15, 2008}
\author{Amiel Ferman\footnote{This article was submitted as part of the first author's PhD Thesis at Bar-Ilan University, Ramat-Gan Israel, 
written under the supervision of Mina Teicher and Tahl Nowik}, Tahl Nowik, Mina Teicher \\ \\ 
Department of Mathematics, Bar-Ilan University \\ Ramat-Gan, Israel }
\maketitle

\begin{abstract}

We prove that an Alexander quandle of prime order 
is generated by any pair of distinct elements. Furthermore, we prove for such a quandle that 
any ordered pair of distinct elements can be sent to any other such pair by an automorphism of the quandle.

\end{abstract}

\section{Introduction}

Quandles were first introduced by Joyce (\cite{Joyce1,Joyce2}) as algebraic
invariants of classical knots and links. For an introduction to the use of
quandles as computable invariants of framed links in $3$-manifolds see \cite{FeRo}.
Another important application of quandles was given by Yetter in \cite{Yetter} as a
means to study braid monodromies of algebraic surfaces. 

Quandles (see Definition~\ref{quand}) can be considered as an abstraction
of groups in the sense that the only binary operation is the analog of the conjugation operation in the group.

In this article we study a particular kind of quandle, called an Alexander quandle (see Definition~\ref{alex}).
It is known that all finite connected quandles with a prime number of elements or with a square of a prime number
of elements are isomorphic to an Alexander quandle (see \cite{EGS} and \cite{Gr} respectively and see 
Definition~\ref{CONNECT} for the notion of a connected quandle). 
Furthermore, studies have shown that a large number of quandles with a small number of elements 
are isomorphic to Alexander quandles (see for example \cite{NeHo}). These facts and the fact
that Alexander quandles are relatively easy to study, due to the arithmetic flavor of their 
definition, serve as a motivation for their study. Finite Alexander quandles were studied in \cite{Nel}
where an arithmetical condition was given as a means to determine whether two such quandles are isomorphic.

In this article we prove that an Alexander quandle of prime order 
is generated by any pair of distinct elements. Furthermore, we prove for such a quandle that 
any ordered pair of distinct elements can be sent to any other such pair by an automorphism of the quandle.

Our paper is organized as follows :

In \textbf{section~\ref{sec quand intro}} we give a short introduction to quandles which is necessary
for the statements of our results and prove a few important lemmas. In \textbf{section~\ref{section alex}}
we prove our main results as follows~: In \textbf{subsection~\ref{subsec finite alex}} we prove a few
general formulae regarding finite Alexander quandles (not necessarily of prime order).
In \textbf{subsection~\ref{subsec alex prime}} we focus on finite Alexander quandles of prime order.
We prove that every Alexander quandle of prime order is generated by any pair of distinct elements. 
Furthermore, we prove for such a quandle that any ordered pair of distinct elements can be sent to 
any other such pair by an automorphism of the quandle.

\section{Basic Definitions and Examples of Quandles}\label{sec quand intro}

In this section we introduce the necessary definitions and properties of quandles
needed for our results on Alexander quandles.

We start with the definition of a quandle.

\begin{definition}[Quandle]\label{quand}
A \textbf{quandle} is a set $X$ with a binary operation written as 
$(a,b) \mapsto a^b$ and satisfies

\begin{list}{}{}
\item[(1)] For every $a,b \in X$ there exists a unique $c\in X$ such 
that $a=c^b$,
\item[(2)] For every $a,b,c\in X$ we have $(a^b)^c = (a^c)^{b^c} $, and
\item[(3)] For every $a\in X$ we have $a^a=a$.
\end{list}

Any subset of elements of a quandle $X$ which is closed under the quandle operation is called a \textbf{subquandle} of $X$.

\hfill $\square$
\end{definition}

\begin{definition}(\textbf{Quandle Homomorphisms})
Let $Q$ and $R$ be quandles. A \textbf{quandle homomorphism} from $Q$ to $R$ is a map $f\colon Q\to R$ 
that satisfies $f(x^y) = f(x)^{f(y)}$ for all $x,y\in Q$. An injective and surjective quandle 
homomorphism from $Q$ to itself is called a quandle \textbf{automorphism}.
\hfill $\square$
\end{definition}

Throughout the paper, we shall use the following

\begin{notation}
For a set $X$, we denote by $\mathbf{F}(X)$ the free group generated by the elements of $X$. \hfill $\bullet$
\end{notation}

Let us now give a few examples of quandles. We begin with

\begin{exa}
The \textbf{trivial quandle} consists of a set $X$ with quandle operation $x^y=x$ for any $x,y\in X$.
\end{exa}

The quintessential example of quandles is

\begin{exa}\label{conj}
Let $G$ be a group. $G$ can be considered as a quandle by letting $a^b \stackrel{def}{=} b^{-1}ab$.
Such a quandle is called a \textbf{Conjugation quandle} and is denoted by $G_{conj}$.
Note that any union of conjugacy classes of $G$ is a subquandle of $G_{conj}$.
Also note that if $G$ is abelian then $G_{conj}$ is a trivial quandle. \hfill $\square$
\end{exa}

Another important example is the following

\begin{exa}\label{alex}
Let $M$ be a $\mathbb Z[t^{\pm 1}]$-module ($t$ acts as some automorphism $\varphi \in Aut_{\mathbb Z}(M)$).
Then $M$ is a quandle, called the \textbf{Alexander quandle}, under the operation $a^b \stackrel{def}{=} t\cdot a + (1-t)\cdot b$ 

\hfill $\square$
\end{exa}

\begin{definition}[\textbf{Generated Subquandle}]
Let $X$ be a quandle and let $A$ be a set of elements of $X$. We say that $A$ generates the subquandle 
$X'$ of $X$ if the minimal (in terms of set containment) subquandle of $X$ which contains $A$ is $X'$. \hfill $\square$
\end{definition}

\begin{definition}[\textbf{Operator Group}]\label{operator}
Let $X$ be a quandle. Let $b\in X$ and consider the map $\rho_b\colon X\to X$ defined as follows

$$ \rho_b(x) = x^b \quad x\in X $$

As a consequence of the first quandle axiom (see Definition~\ref{quand}), the function $\rho_b : x \mapsto x^b$
is a bijection from $X$ to itself. Hence we can write 

$$a^{\bar{b}} = \rho_b^{-1}(a).$$

While $a^{\bar{b}}$ is a well defined member of $X$, it is not necessarily the case that 
$\bar{b} \in X$ but rather that $a^{b\bar{b}} = a^{\bar{b}b} = a$. Hence, if we identify
$\bar{b}$ with $b^{-1}$ we can define the value of an expression of the form $x^w$
where $x\in X$ and $w\in F(X)$ in the following way. Assume $w=w_1\cdots w_n$ where
each $w_i$ is equal to $x_i^{\pm 1}$ for some $x_i\in X$. Then define $x^w$ as follows

$$ x^w = (\cdots((x^{w_1})^{w_2})\cdots)^{w_n} = \rho_{w_n}\circ \cdots \circ \rho_{w_1}(x). $$ 

Denote by $N$ the following normal subgroup of $F(X)$

$$ N = \{ w \in F(X) \big| \ \forall x\in X \ : \ x^w=x  \}  $$

We call the group $F(X)/N$ the \textbf{operator group} of the quandle $X$, and denote it as $Op(X)$.

If $w_1,w_2\in F(X)$ have the same image in $F(X)/N$ (under the natural homomorphism),i.e. if

$$ \forall x\in X \quad x^{w_1} = x^{w_2}  $$

then we say that $w_1$ and $w_2$ are \textbf{operationally equivalent} and denote it as $w_1 \equiv w_2$. \hfill $\square$
\end{definition}

\begin{definition}\label{CONNECT}
A quandle $X$ is \textbf{connected} if the operation group of $X$, $Op(X)$, acts
transitively on $X$. In other words, $X$ is connected if for every $x,y \in X$ there exists $w\in Op(X)$ such that 

$$ x^w = y $$

(see Definition~\ref{operator} for the meaning of the expression $x^w$)
\hfill $\square$
\end{definition}

\begin{lemma}\label{lemma extended axiom}

Let $X$ be a quandle. Then 

\begin{equation}\label{eq quand axiom 2}
(a^b)^c = (a^c)^{b^c} \quad a,b\in X, \quad c\in F(X)
\end{equation}

where, as explained in Definition~\ref{operator}, 
$F(X)$ acts on elements of $X$ by the quandle operation of $X$. 

Furthermore,

\begin{equation}\label{eq quand axiom 2'}
\quad a^{b^c} = a^{c^{-1} b c} \quad a,b\in X, \quad c\in F(X)
\end{equation}

\end{lemma}

\begin{proof}

Let us prove \eqref{eq quand axiom 2} for $c\in F(X)$ by induction
on the length of $c$ as a word in $F(X)$. 

For the induction base, we have to prove \eqref{eq quand axiom 2} for $c$ of length $1$, i.e., for $c^{\pm 1}$ when
$c\in X$.

\eqref{eq quand axiom 2} holds for any $c\in X$
since in this case \eqref{eq quand axiom 2} is just the second quandle axiom of $X$ (see Definition~\ref{quand}).

Now let $c\in X$ and let us prove \eqref{eq quand axiom 2} for $c^{-1}$~: 

\begin{equation}\label{second axiom inverse}
(a^{c^{-1}})^{b^{c^{-1}}} = (((a^{c^{-1}})^{b^{c^{-1}}})^{c})^{c^{-1}} 
\end{equation}

Now note that according to the second quandle axiom (Definition~\ref{quand}), where we put $a^{c^{-1}}$ instead
of $a$ and $b^{c^{-1}}$ instead of $b$, we have

$$ ((a^{c^{-1}})^{b^{c^{-1}}})^{c} = ((a^{c^{-1}})^{c})^{ (b^{c^{-1}})^c  } = a^b. $$

Hence, continuing \eqref{second axiom inverse}, we have

\begin{equation}\label{second axiom inverse'}
(a^{c^{-1}})^{b^{c^{-1}}} = (((a^{c^{-1}})^{b^{c^{-1}}})^{c})^{c^{-1}} = (a^b)^{c^{-1}}. 
\end{equation}

So $(a^{c^{-1}})^{b^{c^{-1}}} = (a^b)^{c^{-1}}$ and the proof of the induction base is complete.

For the induction step, assume (\ref{eq quand axiom 2}) holds for words $c\in F(X)$ of length $<k$ and let
$c = c_1\cdots c_k$, then

$$ (a^b)^c = ((a^b)^{c_1\cdots c_{k-1}})^{c_k} = ((a^{c_1\cdots c_{k-1}})^{b^{c_1\cdots c_{k-1}}})^{c_k} = $$

$$ = ((a^{c_1\cdots c_{k-1}})^{c_k})^{  (b^{c_1\cdots c_{k-1}})^{c_k}  } = (a^c)^{b^c} $$

where the second equality follows from the induction hypothesis and the third equality follows from the second quandle axiom of $X$
(if $c_k = x$ for some $x\in X$) or from \eqref{second axiom inverse'} (if $c_k = x^{-1}$ for some $x\in X$).

Let us now prove (\ref{eq quand axiom 2}) by showing that it is equivalent to $(\ref{eq quand axiom 2'})$. 

Let $a,b\in X$ and $c\in F(X)$. Assume that $(\ref{eq quand axiom 2})$ holds, then

$$ a^{b^c} = ( (a^{c^{-1}})^{c} )^{b^c} = ((a^{c^{-1}})^b)^c = a^{c^{-1} b c} $$

On the other hand, if $(\ref{eq quand axiom 2'})$ holds. Then

$$ (a^c)^{b^c} = a^{c c^{-1} b c} = a^{bc} = (a^b)^c $$

Hence  $(\ref{eq quand axiom 2})$ and axiom $(\ref{eq quand axiom 2'})$ are indeed equivalent.

\end{proof}

\begin{notation}
For a set $X$, we denote by $\mathbf{S}(X)$ the permutation group of the elements of $X$. \hfill $\bullet$
\end{notation}

\begin{lemma}\label{lemma finite quand}
Let $X$ be a quandle. Let $\rho_x \in Aut(X)$ be the automorphism
defined by $\rho_x(y) = y^x$ for any $y\in X$ (see Definition~\ref{operator}).

Then the map $\mu\colon X\to \mathbf{S}(X)_{conj}$ defined by

$$ \mu(x) = \rho_x $$

is a quandle homomorphism.

\end{lemma}

\begin{proof}

$\newline$

For the purposes of this proof let us consider the permutations in $\mathbf{S}(X)$ as right actions on $X$.

Now for any $z\in X$ we have

$$ \mu(x^y)(z) = (z)\rho_{x^y} = z^{x^y} = z^{y^{-1} x y} = (z)\rho_{y}^{-1} \circ \rho_x \circ \rho_{y}   $$

where the third equation follows from \eqref{eq quand axiom 2'} in Lemma~\ref{lemma extended axiom}. Hence

$$ \mu(x^y) = \rho_y^{-1}\circ \rho_x \circ \rho_{y} =  \rho_x^{\rho_y} $$

and so $\mu$ is indeed a quandle homomorphism.

\end{proof}

\section{Alexander Quandles}\label{section alex}

Let us first give the definition of an Alexander quandle~:

\begin{definition}
Given a $\mathbb Z[t^{\pm 1}]$-module $M$, where $t$ acts as an automorphism $\varphi \in Aut_{\mathbb Z}(M)$, 
we define a quandle structure on $M$ by the quandle operation $a^b \stackrel{def}{\equiv} ta + (1-t)b$.
This is called an Alexander quandle. \hfill $\square$
\end{definition}

\begin{rem}\label{alex_conj}
Note that the requirement $t \in Aut_{\mathbb Z}(M)$ is a sufficient condition for the first quandle axiom to hold (see Definition \ref{quand}).

Furthermore, note that requiring $(1-t) \in Aut_{\mathbb Z}(M)$ is a sufficient 
condition for the quandle map $a \mapsto \rho_a$ (see Lemma~\ref{lemma finite quand} ), where $a \in M$ and $\rho_a(b) = b^a$, to be an injection since

$$ \rho_a = \rho_b \Longleftrightarrow   \forall c\in M \quad \rho_a(c) = \rho_b(c) \Longleftrightarrow   $$

$$ \forall c\in M \quad tc + (1-t)a = tc + (1-t)b \Longleftrightarrow  (1-t)(a-b) = 0 $$

hence requiring that $1-t \in Aut_{\mathbb Z}(M)$ is enough to ensure that $\rho_a = \rho_b$ iff $a=b$
which means that the quandle can be considered as conjugacy quandle (see Example \ref{conj}).

Furthermore, if we consider for any $a,b\in M$ the equation $a^c=b$, which is equivalent to

\begin{equation}\label{108.1}
ta + (1-t)c = b
\end{equation}

and if $(1-t)$ is invertible, the solution of \eqref{108.1} is given by $c = (1-t)^{-1}(b - ta)$, so that
$M$ is connected. Hence, if $1-t$ is invertible and $M$ is finite, we can identify $M$
with a subset of a conjugacy class in the permutation group
(where the quandle operation is conjugation in the symmetric group). \hfill $\square$
 
\end{rem}

\subsection{Finite Alexander Quandles}\label{subsec finite alex}

\begin{lemma}\label{power} 

Let $Q$ be a finite Alexander quandle. 
For any $k\in \mathbb Z$ and any $a,b\in Q$ we have

\begin{equation}\label{eq power alex}
a^{b^k} = t^{k}a + (1-t^{k})b.
\end{equation}

Furthermore, if $m$ is the order of $\ t\in Aut(Q)$ then for any $k\in \mathbb Z$

\begin{equation}\label{42.1}
a^{b^k} = a^{b^{k + m}} 
\end{equation}

and

\begin{equation}\label{42.2}
\{ a^{b^k} \ | \ k \in \mathbb Z \} =  \{ a^{b^k} \ | \ k = 0\ldots m-1 \} 
\end{equation}

\end{lemma}

\begin{proof}
$\newline$ Let us first prove \eqref{eq power alex} for $k\in \mathbb N$ by induction on $k$. 
The case $k=0$ clearly holds, and assuming for $k$ we have

$$  a^{b^{k+1}} = (t^{k}a + (1-t^{k})b)^{b} = $$

$$ = t(t^{k}a + (1-t^{k})b) + (1-t)b = t^{k+1}a + (1-t^{k+1})b. $$

Now, according to the first quandle axiom (see Definition~\ref{quand}) and Definition~\ref{operator}, $a^{b^{-1}}$
is the unique element in $Q$ such that

$$ (a^{b^{-1}})^b = a $$

but 

$$ (t^{-1}a + (1-t^{-1})b)^b = t(t^{-1}a + (1-t^{-1})b) + (1-t)b = a $$

Hence 

\begin{equation}\label{42.3}
a^{b^{-1}} = t^{-1}a + (1-t^{-1})b 
\end{equation}

Continuing with the induction proof for negative powers, let $k\in \mathbb N$ and assume the hypothesis for $-k$. 
Then, using \eqref{42.3}, we have

$$  a^{b^{-k-1}} = (t^{-k}a + (1-t^{-k})b)^{b^{-1}} = $$

$$ = t^{-1}(t^{-k}a + (1-t^{-k})b) + (1-t^{-1})b = t^{-k-1}a + (1-t^{-k-1})b. $$

To prove \eqref{42.1}, simply note that for any $k$, according to \eqref{eq power alex} and since $m$ is the order of $t$, we have that

$$ a^{b^k} = t^{k}a + (1-t^{k})b = t^{k + m} a + (1-t^{k + m})b = a^{b^{k + m}} $$

\eqref{42.2} is now a direct corollary from \eqref{42.1}.

\end{proof}

\begin{lemma}\label{sum_cycle} Let $M$ be a finite Alexander quandle and let $m$ be the order of $\ t\in Aut_{\mathbb Z}(M)$, 
furthermore assume that $1-t\in Aut_{\mathbb Z}(M)$ then for any $a,b \in M$ and any $k\in \mathbb N$

$$ a + a^b + \cdots + a^{b^{k}} = (1 + t + \cdots + t^k)a + (k - t - \cdots - t^k )b $$

\flushleft and in particular 

$$ a + a^b + \cdots + a^{b^{m-1}} = m\cdot b $$
\end{lemma}

\begin{proof}
Let us first prove by induction that for $k\in \mathbb N$ we have

$$ a + a^b + \cdots + a^{b^{k}} = (1 + t + \cdots + t^k)a + (k - t - \cdots - t^k)b $$

\flushleft this clearly holds for $k=0$, now assuming for $k$ and using Lemma~\ref{power} we have

$$ a + a^b + \cdots + a^{b^{k}} + a^{b^{k+1}} = $$ 
$$ = (1 + t + \cdots + t^k)a + (k - t - \cdots - t^k)b + t^{k+1}a + (1-t^{k+1})b = $$
$$ = (1 + t + \cdots + t^{k+1})a + (k+1 - t - \cdots - t^{k+1})b$$

\flushleft now since (recall that $m$ is the order of $t$)

$$ (1-t)(1 + \cdots + t^{m-1}) = 1-t^m = 0 $$

\flushleft hence by our assumption that $1-t \in Aut_{\mathbb Z}(M)$ we must have that

$$ 1 + \cdots + t^{m-1} = 0 $$

\flushleft and this means that 

$$  a + a^b + \cdots + a^{b^{m-1}} = (m-1+1)\cdot b = m\cdot b $$

\end{proof}

\begin{lemma}\label{rels}
Let $M$ be an Alexander quandle, then for any $a,b \in M$ and $k_i \in \mathbb Z$ we have that~:

For $n$ odd~:

\begin{eqnarray}
a^{b^{k_1} a^{k_2} \cdots b^{k_{n-2}} a^{k_{n-1}} b^{k_n}} = (\sum_{i=1 \ldots n} (-1)^{i+1} t^{ k_i + \cdots + k_n } )(a-b) + b\label{eq rels 1} 
\end{eqnarray}

For $n$ even~:

\begin{eqnarray}
a^{b^{k_1} a^{k_2} \cdots b^{k_{n-1}} a^{k_n}} = (\sum_{i=1 \ldots n} (-1)^{i+1} t^{ k_i + \cdots + k_n } )(a-b) + a\label{eq rels 3} 
\end{eqnarray}

\end{lemma}

\begin{proof}

Let us prove \eqref{eq rels 1} and \eqref{eq rels 3} by induction on $n$.
The case $n=1$ follows from \eqref{eq power alex} in Lemma~\ref{power}.


Assume \eqref{eq rels 1} and \eqref{eq rels 3} hold for $< n$, for $n$ odd. Using Lemma~\ref{power}, we have

$$ (a^{b^{k_1} a^{k_2} \cdots b^{k_{n-2}} a^{k_{n-1}}})^{b^{k_n}} = $$

$$ = t^{k_n} ( (\sum_{i=1 \ldots n-1} (-1)^{i+1} t^{ k_i + \cdots + k_{n-1} } )(a-b) + a) + (1-t^{k_n})b = $$

(since $n$ is odd, $(-1)^{n+1}=1$ and so)


$$ = (\sum_{i=1 \ldots n} (-1)^{i+1} t^{ k_i + \cdots + k_n } )(a-b) + b $$

and using the last equality we have


$$ (a^{b^{k_1} a^{k_2} \cdots a^{k_{n-1}} b^{k_{n}}})^{a^{k_{n+1}}} = $$

$$ = t^{k_{n+1}}((\sum_{i=1 \ldots n} (-1)^{i+1} t^{ k_i + \cdots + k_{n} } )(a-b) + b) + (1-t^{k_{n+1}})a = $$

(since $n+1$ is even, $(-1)^{n+2}=-1$ and so)


$$ = (\sum_{i=1 \ldots n+1} (-1)^{i+1} t^{ k_i + \cdots + k_{n+1} } )(a-b) + a. $$


\end{proof}

\subsection{Alexander Quandles of Prime Order}\label{subsec alex prime}

In this section we prove our two main results regarding
Alexander quandles of prime order. We determine that each such quandle is generated
by (any) two elements and we also describe its set of quandle automorphisms.

Fix $p$ to be some prime number. 

\begin{rem}\label{rem alex p}
In what follows we will consider $\mathbb Z_p$ as an Alexander quandle. Recall that
each group automorphism of $\mathbb Z_p$ is a multiplication by a number in $\mathbb Z_p^*$ and
so we will sometimes consider $t$ as a number in $\mathbb Z_p^*$. \hfill $\bullet$
\end{rem}

\begin{lemma}\label{size_cycle}
Let $p\in \mathbb N$ be some prime number and consider $Q=\mathbb Z_p$ as an Alexander quandle with
$t\in \mathbb Z_p^*$, $\ t \neq 1$ (see Remark~\ref{rem alex p}), and let $m$ be the order of $\ t\in Aut(\mathbb Z_p)$. Then
for any $a,b\in Q$, $a\neq b$,  we have

$$ |\{ a^{b^k} \ | \ k \in \mathbb N \}| = m $$

\end{lemma}

\begin{proof}
By Lemma~\ref{power} we have that 

$$ \{ a^{b^k} \ | \ k \in \mathbb N \} =  \{ a^{b^k} \ | \ k = 0\ldots m-1 \} $$

now suppose $a^{b^i}=a^{b^j}$ for some $0\leq i,j < m$ then

$$ t^i a + (1-t^i)b = t^j a + (1-t^j)b $$

\flushleft or

$$ (t^i - t^j)(a-b)=0 $$

but according to our assumption $0\leq i,j < m$ and $m$ is the order of $t$, hence it must be
that $i=j$ for otherwise $t^i-t^j\neq 0$ would be a zero divisor in $\mathbb Z_p$.

\end{proof}

\begin{theorem}\label{gens}
Let $p\in \mathbb N$ be some prime number and consider $\mathbb Z_p$ as an Alexander quandle with
$t\in \mathbb Z_p^*$, $\ t \neq 1$ (see Remark~\ref{rem alex p}), and let $m$ be the order of $\ t\in Aut_{\mathbb Z}(\mathbb Z_p)$. 
Then any two elements $a,b\in Q$, $a\neq b$,  generate $Q$.
\end{theorem}

\begin{proof}

Let $a,b \in Q$, $a\neq b$. If $p=2$ then the claim is trivial. Assume then that $p>2$.

Let us prove that for any $c\in \mathbb Z_p$, $c\neq a$, there exists an even $n\in \mathbb N$ 
and $k_1,\ldots,k_n\in \mathbb N$ such that

\begin{equation}\label{41.1}
a^{b^{k_1} a^{k_2} \cdots b^{k_{n-1}} a^{k_n}} = c 
\end{equation}

According to Lemma~\ref{rels}, \eqref{41.1} is equivalent to 

\begin{equation}\label{41.2}
(\sum_{i=1 \ldots n} (-1)^{i+1} t^{ k_i + \cdots + k_n } )(a-b) + a = c 
\end{equation}

It is enough to show that for any $d\in \mathbb Z_p^*$ there exists an even $n\in \mathbb N$ and $k_1,\ldots,k_n\in \mathbb N$ such that

$$ \sum_{i=1 \ldots n} (-1)^{i+1} t^{ k_i + \cdots + k_n } = d $$ 

for then we can take $d=(c-a)(a-b)^{-1}$ in \eqref{41.2} ($d\neq 0$ since according to our assumption $c\neq a$).

Let $n$ be some even number and put $k_i = (-1)^i$ for each $i=1,\ldots,n$. Then we have that

$$ \sum_{i=1 \ldots n} (-1)^{i+1} t^{ k_i + \cdots + k_n } = \sum_{i=1 \ldots n} (-1)^{i+1} t^{ (-1)^{i} + \cdots + (-1)^{n} } = $$

$$ = \frac12n - \frac12nt = \frac12n(1-t) $$

So for any even $n\in \mathbb N$ we have

\begin{equation}\label{41.3}
\sum_{i=1 \ldots n} (-1)^{i+1} t^{ (-1)^i + \cdots + (-1)^n } = \frac12n(1-t)
\end{equation}

Now note that for any $e\in \mathbb Z_p$ there exists an integer $n'\in \mathbb N$ such that

\begin{equation}\label{41.4}
2n'\equiv e(\text{mod} \ p) 
\end{equation}

this follows from the fact that \eqref{41.4} is equivalent to 

$$ n'\equiv 2^{-1}e(\text{mod} \ p)  $$

and by assumption $p$ is a prime greater than $2$. In other words, for any $e\in \mathbb Z_p$, there
exists an even $n\in \mathbb N$ such that $n\equiv e(\text{mod} \ p)$.

This means that in the following equation

$$ \frac12n(1-t) \equiv d(\text{mod} \ p) $$

we can choose an even $n$ such that $n\equiv 2(1-t)^{-1}d(\text{mod} \ p)$ and so, putting
this $n$ in \eqref{41.3} we have that

$$ \sum_{i=1 \ldots n} (-1)^{i+1} t^{ (-1)^i + \cdots + (-1)^n } = \frac12n(1-t) = d $$

which completes the proof.

\end{proof}

\begin{theorem}\label{prop alex aut}
Consider $Q = \mathbb Z_p$ as an Alexander quandle for some prime number $p$ with $t\in \mathbb Z_p^*$, $\ t \neq 1$ 
(see Remark~\ref{rem alex p}). For any $a,b,c,d\in Q$
such that $a\neq b$ and $c\neq d$ there exists a unique quandle automorphism $\alpha\in Aut(Q)$ such that
$\alpha(a) = c$ and $\alpha(b) = d$.
\end{theorem}

\begin{proof}

According to Theorem~\ref{gens}, $Q$ is generated by $a$ and $b$. Hence, it is enough to define $\alpha$
on words generated by $a$ and $b$. 






Let $e$ equal $a$ or $b$ and let $x\in F(\{a,b\})$ where $x=x_1\cdots x_n$ and
$x_i=a^{\pm 1}$ or $x_i=b^{\pm 1}$ for each $i=1,\ldots,n$. Then we define $\alpha$ as follows (see also Definition~\ref{operator})

\begin{equation}\label{44.1}
\alpha(e^x) = \alpha(e^{x_1\cdots x_n}) = \alpha(e)^{\alpha(e_1)^{\epsilon_1}\cdots \alpha(e_n)^{\epsilon_n}} = \alpha(e)^{\alpha(x)}
\end{equation}

where $\epsilon_i=\pm 1$ when $x_i=e_i^{\pm 1}$ where each $e_i$ is either $a$ or $b$. 
In other words, on the operational level (see Definition~\ref{operator})
we consider $\alpha$ as the group homomorphism $F(\{a,b\})\to F(\{c,d\}$ sending $a$ to $c$ and $b$ to $d$.

First, let us show that $\alpha$ is well defined.
We need to check that for each two words $x$ and $y$ generated by $a$ and $b$ the equation $x=y$ in $Q$ implies that
$\alpha(x)=\alpha(y)$.

Consider then a typical word $x$ generated by $a$ and $b$. There are four possible forms for such a word $x$~:

\begin{eqnarray}
&& a^{b^{k_1} a^{k_2} \cdots b^{k_{n-2}} a^{k_{n-1}} b^{k_n}}\label{eq ab1} \\
&& b^{a^{k_1} b^{k_2} \cdots a^{k_{n-2}} b^{k_{n-1}} a^{k_n}}\label{eq ab2} \\
&& a^{b^{k_1} a^{k_2} \cdots b^{k_{n-1}} a^{k_n} }\label{eq ab3} \\
&& b^{a^{k_1} b^{k_2} \cdots a^{k_{n-1}} b^{k_n} }\label{eq ab4}
\end{eqnarray}

According to Lemma~\ref{rels}, the words in \eqref{eq ab1} to \eqref{eq ab4} are equal, respectively, to

\begin{eqnarray}
(\sum_{i=1 \ldots n} (-1)^{i+1} t^{ k_i + \cdots + k_n } )(a-b) + b\label{eq rels 1'} \\
(\sum_{i=1 \ldots n} (-1)^{i+1} t^{ k_i + \cdots + k_n } )(b-a) + a\label{eq rels 2'} \\ 
(\sum_{i=1 \ldots n} (-1)^{i+1} t^{ k_i + \cdots + k_n } )(a-b) + a\label{eq rels 3'} \\
(\sum_{i=1 \ldots n} (-1)^{i+1} t^{ k_i + \cdots + k_n } )(b-a) + b\label{eq rels 4'}  
\end{eqnarray}

Let us rewrite these expressions as follows~:

\begin{eqnarray}
(\sum_{i=1 \ldots n} (-1)^{i+1} t^{ k_i + \cdots + k_n } - 1 )(a-b) + a\label{eq rels 1''} \\
(-\sum_{i=1 \ldots n} (-1)^{i+1} t^{ k_i + \cdots + k_n } )(a-b) + a\label{eq rels 2''} \\ 
(\sum_{i=1 \ldots n} (-1)^{i+1} t^{ k_i + \cdots + k_n } )(a-b) + a\label{eq rels 3''} \\
(-1 -\sum_{i=1 \ldots n} (-1)^{i+1} t^{ k_i + \cdots + k_n } )(a-b) + a\label{eq rels 4''}  
\end{eqnarray}

It is now clear to see that by equating two expressions from the forms appearing in \eqref{eq rels 1''} to \eqref{eq rels 4''}
we get an equation of the form

$$ K(a-b) + a = K'(a-b) + a $$

or

\begin{equation}\label{eq final rel}
(K-K')(a-b) = 0
\end{equation}

where $K,K'\in \mathbb Z_p$ are independent of $a$ and $b$ and are only depended on the indices $k_1,\ldots,k_n$ 
appearing in the powers of the expressions in \eqref{eq ab1} to \eqref{eq ab4}. But since $a\neq b$ we have that \eqref{eq final rel}
is equivalent to

$$ K-K' \equiv 0(mod\ p) $$

which in turn is equivalent to (recall that $c\neq d$ by assumption)

$$ (K-K')(c-d) = 0 $$

or

$$ K(c-d) + c = K'(c-d) + c $$

and this is the same kind of equation we started from only that we use $c$ instead of $a$ and $d$ instead of $b$.

Thus, we see that if $x$ and $y$ are words generated by $a$ and $b$ then the equality $x=y$
holds in $Q$ iff the equality $x'=y'$ holds where $x'$ and $y'$ are the same as $x$ and $y$
only that the letter $a$ has been replaced by $c$ and the letter $b$ has been replaced by $d$.
But since $\alpha(x)=x'$ and $\alpha(y)=y'$, as explained above, this concludes the proof that $\alpha$ is well defined.

Let us show that $\alpha$ in onto. Let $z$ be a word generated by $c$ and $d$ (since $c\neq d$ we have
that $c$ and $d$ generate $Q$ by Theorem~\ref{gens} and so $z$ could be any element in $Q$). Then
clearly by the definition of our extension \eqref{44.1} we have that $\alpha(z')=z$
where $z$ is the same as $z'$ only that the letter $c$ has been replaced by $a$ and $d$ has been
replaced by $b$.

Since $Q$ is finite then $\alpha$ must also be one-to-one.

Let us now show that $\alpha$ is a quandle homomorphism. Let $e_1$ and $e_2$ be either $a$ or $b$ and let $x,y\in F(\{a,b\})$. 
Then using \eqref{44.1} and \eqref{eq quand axiom 2'} in Lemma~\ref{lemma extended axiom} we have

$$ \alpha((e_1^x)^{e_2^y}) = \alpha(e_1^{x y^{-1} e_2 y}) = \alpha(e_1)^{\alpha(x) \alpha(y)^{-1} \alpha(e_2) \alpha(y)} = $$

$$ = (\alpha(e_1)^{\alpha(x)})^{\alpha(e_2)^{\alpha(y)}} = \alpha(e_1^x)^{\alpha(e_2^y)}. $$

Now \eqref{44.1} define $\alpha$ on $Q$ and since
these equations are also necessary in order to ensure that $\alpha$ is a quandle homomorphism, we
conclude that any other homomorphism sending $a$ to $c$ and $b$ to $d$ must in fact be identical to $\alpha$.

Summing up, we have shown that there is a unique quandle automorphism, namely $\alpha$, such that $\alpha(a)=c$ and $\alpha(b)=d$.
This concludes the proof.

\end{proof}

In fact, Theorems~\ref{gens} and~\ref{prop alex aut} apply to any connected finite quandle of prime order~: It
was proved in \cite{EGS} that any connected finite quandle of prime order is isomorphic to an Alexander quandle
$X$ with $t\neq 1$.


\addcontentsline{toc}{chapter}{Bibliography}

\begin{thebibliography}{99}


\bibitem[E07]{Ei} M. Eisermann, \emph{Quandle Coverings and their Galois Correspondence},
Preprint, \texttt{arXiv:math.GT/0612459}.

\bibitem[EGS01]{EGS} P. Etingof, R. Guralnik, A. Soloviev, \emph{Indecomposable set-theoretical solutions to the Quantum
Yang-Baxter Equation on a set with prime number of elements}, J. Algebra {\bf 242} (2001), 709-719.

\bibitem[FR92]{FeRo} Fenn R., Rourke C., \emph{Racks and Links in Codimension two}, 
Journal of Knot Theory and its Ramifications, \textbf{Volume 1} (1992), pages 343-406


\bibitem[Gr02]{Gr} M. Grana, \emph{Indecomposable Racks of order $p^2$},
Preprint, \texttt{arXiv:math.QA/0203157}

\bibitem[Joy79]{Joyce1}
D.E. Joyce.
\newblock {\em An algebraic approach to symmetry with applications to knot
  theory}.
\newblock PhD thesis, University of Pennsylvania, 1979.


\bibitem[Joy82]{Joyce2}
D.E. Joyce.
\newblock A classifying invariant of knots, the knot quandle.
\newblock {\em Journal of Pure and Applied Algebra}, 23:37--65, 1982.


\bibitem[LoRo]{LoRo04} P. Lopes, D. Roseman, {\em On finite racks and quandles },
may be found on http://arXiv.org, math.GT/0412487.

\bibitem[Ne03]{Nel} S. Nelson, \emph{Classification of Finite Alexander Quandles}, \texttt{arXiv.org:math.GT/0202281}

\bibitem[NeHo05]{NeHo} S. Nelson, B. Ho, \emph{Matrices and Finite Quandles}, 
Homology Homotopy Appl. Volume 7, Number 1 (2005), 197-208


\bibitem[Ye06]{Yetter} Yetter D., \emph{Quandles and Monodromy}, (arXiv: math.GT/0205162)




\end{thebibliography}
\end{document}